\documentclass[12pt]{amsart}

\usepackage{ucs}

\usepackage{amssymb}
\usepackage{amsthm}
\usepackage{amsmath}
\usepackage{latexsym}
\usepackage[cp1251]{inputenc}
\usepackage{graphicx}
\usepackage{wrapfig}
\usepackage{caption}
\usepackage{subcaption}
\usepackage{indentfirst}
\usepackage[left=2.6cm,right=2.6cm,top=2.6cm,bottom=2.6cm,bindingoffset=0cm]{geometry}
\usepackage{enumerate}

\DeclareMathOperator{\aut}{Aut}

\DeclareMathOperator{\id}{id}

\DeclareMathOperator{\inv}{Inv}

\DeclareMathOperator{\Span}{Span}

\DeclareMathOperator{\sym}{Sym}
\DeclareMathOperator{\rad}{rad}

\DeclareMathOperator{\OO}{O}

\makeatletter 
\def\@seccntformat#1{\csname the#1\endcsname. } 
\def\@biblabel#1{#1.}

\makeatother

\title{Infinite family of nonschurian separable association schemes}

\author{Grigory Ryabov}
\address{Sobolev Institute of Mathematics, Novosibirsk, Russia}
\address{Novosibirsk State University, Novosibirsk, Russia}
\email{gric2ryabov@gmail.com}

\thanks{The work is supported by Russian Scientific Fund (project No.~19-11-00039)}

\date{}

\newtheorem*{theo1}{Theorem 1}

\newtheorem*{ques}{Question}

\newtheorem{prop}{Proposition}[section]
\newtheorem{lemm}[prop]{Lemma}

\begin{document}

\vspace{\baselineskip}
\vspace{\baselineskip}

\vspace{\baselineskip}

\vspace{\baselineskip}

\begin{abstract}
An infinite family of nonschurian separable association schemes is constructed.
\\
\\
\textbf{Keywords}: Coherent configurations, association schemes, isomorphisms.
\\
\textbf{MSC}:05E30, 05C60, 20B35.
\end{abstract}

\maketitle

\section{Introduction}
A \emph{coherent configuration} $\mathcal{X}$ on a finite set $\Omega$ can be thought as a special partition of $\Omega\times \Omega$ for which the diagonal $1_{\Omega}$ of $\Omega\times \Omega$ is a union of classes.  The number $|\Omega|$ is called the \emph{degree} of $\mathcal{X}$. If the diagonal is exactly a class of the partition then $\mathcal{X}$ is an \emph{association scheme}. The notion of a coherent configuration goes back to Higman who studied permutation groups via analyzing invariant binary relations~\cite{Hig} and Weisfeiler and Leman in connection with the graph isomorphism problem~\cite{WeisL}.  For more details on coherent configurations we refer the readers to~\cite{CP}. 

There are two crucial properties of coherent configurations, namely, \emph{schurity} and \emph{ separability}. We explain them below. The problem of determining whether a given coherent configuration is schurian (separable) is among the most fundamental problems in the theory of coherent configurations.

Let $K$ be a permutation group on the set $\Omega$. Then the partition of $\Omega\times \Omega$ into the $2$-orbits of $K$ defines a coherent configuration $\inv(K)$ on $\Omega$. A coherent configuration $\mathcal{X}$ on $\Omega$ is called \emph{schurian} if $\mathcal{X}=\inv(K)$ for some $K\leq \sym(\Omega)$, or equivalently, if $\mathcal{X}=\inv(\aut(\mathcal{X}))$. The mapping $\mathcal{X}\mapsto \aut(\mathcal{X})$ is a Galois correspondence between schurian coherent configurations and the $2$-closed permutation groups~\cite[Theorem~2.2.8]{CP}.

Schurian coherent configuration are usually studied in the frames of the permutation group theory.  However, the most of coherent configurations are nonschurian. By this reason, the theory of coherent configurations is sometimes called ``group theory without groups''. One of the motivations to study shurian coherent configurations comes from the graph isomorphism problem which is equivalent to finding the orbits of the automorphism group of a given coherent configuration. For a schurian coherent configuration, the latter can be done efficiently (see~\cite[Section~2.2.3]{CP}).

Let $\mathcal{X}=(\Omega,S)$ and $\mathcal{X}^{\prime}=(\Omega^{\prime},S^{\prime})$ be coherent configurations. An \emph{algebraic isomorphism} between $\mathcal{X}$ and $\mathcal{X}^{\prime}$ is a bijection from $S$ to $S^{\prime}$ that preserves intersection numbers. A \emph{combinatorial isomorphism} between $\mathcal{X}$ and $\mathcal{X}^{\prime}$ is a bijection from $\Omega$ to $\Omega^{\prime}$ that maps the basis relations of $\mathcal{X}$ to the basis relations of $\mathcal{X}^{\prime}$. Every combinatorial isomorphism induces the algebraic one. However, the converse statement does not hold in general. A coherent configuration $\mathcal{X}$ is called \emph{separable} if every algebraic isomorphism from $\mathcal{X}$ to another coherent configuration is induced by a combinatorial isomorphism. In this case $\mathcal{X}$ is determined up to isomorphism by the tensor of its intersection numbers. It follows that if $\mathcal{X}$ is constructed from a graph by the Weisfeiler-Leman algorithm then the isomorphism between this graph and any other graph can be verified efficiently in a pure combinatorial way (see~\cite[Section~4.6.1]{CP}).

It seems natural to ask how the schurity and separability are related. Actually, there are infinite families of nonschurian nonseparable coherent configurations as well as schurian separable coherent configurations~\cite{CP}. Infinite families of schurian nonseparable coherent configurations arise from finite affine and projective planes~\cite[Sections~2.5.1-2.5.2]{CP}. However, to the best of the author knowledge, no infinite family of nonschurian separable coherent configurations is known. The following question was asked in~\cite[p.~65]{CP}.

\begin{ques}
Whether there exists an infinite family of nonschurian separable coherent configurations?
\end{ques}

We give an affirmative answer to this question. More precisely, we prove the following theorem.

\begin{theo1}\label{main}
For every prime $p\geq 5$, there exists a nonschurian association scheme of degree $4p^2$ which is separable.
\end{theo1}

The author would like to thank prof. I. Ponomarenko for drawing his attention to the problem and the anonymous referee for valuable comments which help the author to improve the text significantly.

\section{Construction}
In this section we describe an infinite family of association schemes. In the next section it will be proved that they are nonschurian and separable. The schemes from this family are Cayley schemes. To make the description clearer, we describe firstly $S$-rings (Schur rings) corresponding to them. In what follows, we freely use basic facts concerned with coherent configurations and $S$-rings. For a background of coherent configurations and $S$-rings, we refer the readers to the monograph~\cite{CP} and the paper~\cite{Ry} respectively, where  necessary statements, definitions, and notations can be found.

\subsection{$S$-ring}
The cyclic group of order $n$ is denoted by $C_n$. Let $p\geq 5$ be a prime and $G=A\times P$, where $A\cong C_2\times C_2$ and $P\cong C_p\times C_p$. Let $a_1$, $a_2$, and $a_3$ be nontrivial elements of $A$ and $P_1$, $P_2$, $P_3$ pairwise distinct subgroups of $P$ of order~$p$. Put $I=\{1,2,3\}$, $X_i=P_ia_i$ for each $i\in I$, and
$$\mathcal{S}=\{\{g\},~X_ig:~g\in P,~i\in I\},~\text{and}~\mathcal{A}=\Span_{\mathbb{Z}}\{\underline{X}:~X\in \mathcal{S}\},$$
where $\underline{X}=\sum_{x\in X} {x}$. For an arbitrary $X\subseteq G$, we put $\rad(X)=\{g\in G:~Xg=gX=X\}$ and $X^{-1}=\{x^{-1}:~x\in X\}$.

\begin{lemm}\label{sring}
In the above notations, $\mathcal{A}$ is a commutative $S$-ring over $G$.
\end{lemm}
\begin{proof}
We need to verify that $\mathcal{S}$ is a Schur partition of $G$, i.e. $\{1_G\}\in \mathcal{S}$, where $1_G$ is the identity of $G$, $\mathcal{S}=\mathcal{S}^{-1}=\{X^{-1}:~X\in \mathcal{S}\}$, and $\underline{X}\underline{Y}\in \mathcal{A}$ for every $X,Y\in \mathcal{S}$.

Clearly, $\mathcal{S}$ is a partition of $G$ such that $\{1_G\}\in \mathcal{S}$. Let $X\in \mathcal{S}$. If $X\subseteq P$ then $X=\{g\}$ for some $g\in P$. In this case $X^{-1}=\{g^{-1}\}$ and $g^{-1}\in P$. So $X^{-1}\in \mathcal{S}$. One can see that $X_i=X_i^{-1}\in \mathcal{S}$ for each $i\in I$. If $X=X_ig$ for some $i\in\{1,2,3\}$ and $g\in P$ then $X^{-1}=X_ig^{-1}\in\mathcal{S}$. Thus, $\mathcal{S}=\mathcal{S}^{-1}$.

Let $X,Y\in \mathcal{S}$. If $|X|=1$ or $|Y|=1$ then $XY\in \mathcal{S}$. Let $|X|=|Y|=p$. If $\rad(X)=\rad(Y)$ then $XY\subseteq P$ and hence $\underline{X}\underline{Y}\in \mathcal{A}$. If $\rad(X)\neq \rad(Y)$ then $XY$ is a $P$-coset. Again, $\underline{X}\underline{Y}\in \mathcal{A}$. Thus, $\mathcal{S}$ is a Schur partition and hence $\mathcal{A}$ is an $S$-ring. Since $G$ is abelian, $\mathcal{A}$ is commutative.
\end{proof}

The \emph{thin radical} of an $S$-ring is defined to be the subgroup generated by all basic sets of size~$1$. 

\begin{lemm}\label{propertyring}
In the above notations, $\mathcal{A}$ satisfies the following properties.
\\
\\
$(A1)$ The thin radical $\OO_{\theta}(\mathcal{A})$ of $\mathcal{A}$ is equal to $P$ and $\mathcal{A}_{G/P}\cong \mathbb{Z}C_2^2$.
\\
\\
$(A2)$ If $i\in I$ then $X_i=X_i^{-1}$, $|X_i|=|\rad(X_i)|=p$, $\rad(X_i)\leq P$, $\rad(X_i)\cong C_p$, and $\rad(X_i)\cap \rad(X_j)=\{1_G\}$ whenever $j\neq i$.
\\
\\
$(A3)$ If $i,j\in I$, $i\neq j$, $T_i=X_ig_i$, $T_j=X_jg_j$ for some $g_i,g_j\in P$, and $T\in \mathcal{S}$ then
$$c_{T_iT_j}^{T}>0 \Leftrightarrow T=X_kg_k~\text{for}~k\in I\setminus\{i,j\}~\text{and some}~g_k\in P.$$
\end{lemm}

\begin{proof}
The property $(A1)$ follows directly from the definition of $\mathcal{S}$. The property $(A2)$ follows from the definition of $\mathcal{S}$ and the observation that $\rad(X_i)=P_i$. Let us check $(A3)$. Note that $T_iT_j\subseteq Pa_k$, where $k\in I\setminus\{i,j\}$. Clearly, $c_{T_iT_j}^{T}>0$ if and only if $T\subseteq T_iT_j\subseteq Pa_k$. The latter inclusion is equivalent to the fact that $T=X_kg_k$ for some $g_k\in P$.
\end{proof}

Let $i\in I$ and $Y_i=Pa_i$. It is easy to see that $Y_i=\bigcup \limits_{g\in P} X_ig$. Put
$$\mathcal{S}_i=\{\{g\}, X_jg, X_kg, Y_i:~g\in P\}~\text{and}~\mathcal{A}_i=\Span_{\mathbb{Z}}\{\underline{X}:~X\in \mathcal{S}_i\},$$
where $\{j,k\}=I\setminus \{i\}$.
Let $(i,j)\in J=\{(1,2),(1,3),(2,3)\}$. Put
$$\mathcal{S}_{ij}=\{\{g\},X_kg, Y_i,Y_j:~g\in P\}~\text{and}~\mathcal{A}_{ij}=\Span_{\mathbb{Z}}\{\underline{X}:~X\in \mathcal{S}_{ij}\},$$
where $\{k\}=I\setminus \{i,j\}$.
Finally, put
$$\mathcal{S}_0=\{\{g\},Y_1,Y_2,Y_3:~g\in P\}~\text{and}~\mathcal{A}_0=\Span_{\mathbb{Z}}\{\underline{X}:~X\in \mathcal{S}_{ij}\}.$$
 
\begin{lemm}\label{orderring}
In the above notations, the following statements hold.
\\
\\
$(1)$ If $i,j\in I$ such that $i\neq j$ then $\mathcal{S}=\{X_i\cap X_j:~X_i\in \mathcal{S}_i,~X_j\in \mathcal{S}_j\}$.
\\
\\
$(2)$ If $(j,k),(l,m)\in J$ such that $\{j,k\}\cap \{l,m\}=\{i\}$ then $\mathcal{S}_i=\{X_{jk}\cap X_{lm}:~X_{jk}\in \mathcal{S}_{jk},~X_{lm}\in \mathcal{S}_{lm}\}$.
\end{lemm}

\begin{proof}
The statement of the lemma immediately follows from the definitions of $\mathcal{S}_i$ and $\mathcal{S}_{ij}$.
\end{proof}

\begin{lemm}\label{fusionring}
In the above notations, the following statements hold.
\\
\\
$(1)$ For every $i\in I$, $\mathcal{A}_i$ is an $S$-ring over $G$ such that $\mathcal{A}_i\leq \mathcal{A}$ and $\mathcal{A}_i\cong (\mathbb{Z}C_p\wr \mathbb{Z}C_2)\otimes (\mathbb{Z}C_p\wr \mathbb{Z}C_2)$, where $\wr$ and $\otimes$ denote the wreath and tensor products of $S$-rings respectively.
\\
\\
$(2)$ For every $(i,j)\in J$, $\mathcal{A}_{ij}$ is an $S$-ring over $G$ such that $\mathcal{A}_{ij}\leq \mathcal{A}_i$ and $\mathcal{A}_{ij}\leq \mathcal{A}_j$.
\\
\\
$(3)$ $\mathcal{A}_0$ is an $S$-ring over $G$ such that $\mathcal{A}_0\leq \mathcal{A}_{ij}$ for every $(i,j)\in J$ and $\mathcal{A}_0\cong\mathbb{Z}C_p^2\wr \mathbb{Z}C_2^2$.
\end{lemm}

\begin{proof}
Let $i\in I$. From the definition of $Y_i$ it follows that: $(1)$ $Y_i=Y_i^{-1}$; $(2)$ $Y_ig=Y_i$ for every $g\in P$; $(3)$ if $j\neq i$ then $\underline{Y_i}\underline{X_j}=\underline{P}\underline{P_i}a_ia_j=p\underline{P}a_k=p\underline{Y_k}$, where $k\in I\setminus \{i,j\}$, and $\underline{Y_i}\underline{X_j}\subseteq P$ if $i=j$; $(4)$ if $i\neq j$ then $\underline{Y_i}\underline{Y_j}=p^2\underline{Y_k}$, where $k\in I\setminus \{i,j\}$, and $\underline{Y_i}\underline{Y_j}\subseteq P$ if $i=j$. This implies that each of the partitions $\mathcal{S}_i$, $i\in I$, $\mathcal{S}_{ij}$, $(i,j)\in J$, $\mathcal{S}_0$ defines an $S$-ring over $G$. Due to the definitions of $\mathcal{S}_i$, $\mathcal{S}_{ij}$, $\mathcal{S}_0$, we obtain $\mathcal{A}_i\leq \mathcal{A}_i$ for every $i\in I$, $\mathcal{A}_{ij}\leq \mathcal{A}_i$ and $\mathcal{A}_{ij}\leq \mathcal{A}_j$ for every $(i,j)\in J$, and $\mathcal{A}_{0}\leq \mathcal{A}_{ij}$ for every $(i,j)\in J$.

Again, let $i\in I$. Put $U_j=\langle X_j \rangle=P_j\times \langle a_j \rangle$ and $U_k=\langle X_k \rangle=P_k\times \langle a_k \rangle$, where $\{j,k\}=I\setminus \{i\}$. The groups $U_j$ and $U_k$ are $\mathcal{A}_i$-subgroups because they are generated by basic sets of $\mathcal{A}_i$. One can see that $U_j\cap U_k=\{1_G\}$ and hence $G=U_j\times U_k$. Since $X_jX_k=P_ja_jP_ka_k=Y_i$, we conclude that 
$$\mathcal{A}_i=\mathcal{A}_{U_j}\otimes \mathcal{A}_{U_k}.~\eqno(1)$$
Note that $(\mathcal{A}_i)_{P_j}\cong (\mathcal{A}_i)_{P_k}\cong \mathbb{Z}C_p$, $\rad(X_j)=P_j$, and $\rad(X_k)=P_k$. So $\mathcal{A}_{U_j}\cong \mathcal{A}_{U_k}\cong \mathbb{Z}C_p\wr \mathbb{Z}C_2$. Together with Eq.~(1), this yields that $\mathcal{A}_i\cong (\mathbb{Z}C_p\wr \mathbb{Z}C_2)\otimes (\mathbb{Z}C_p\wr \mathbb{Z}C_2)$.

By the definition of $\mathcal{A}_0$, we have $(\mathcal{A}_0)_P\cong \mathbb{Z}C_p^2$ and each basic set of $\mathcal{A}_0$ outside $P$ is a $P$-coset. Thus, $\mathcal{A}_0=(\mathcal{A}_0)_P\wr (\mathcal{A}_0)_{G/P}\cong \mathbb{Z}C_p^2\wr \mathbb{Z}C_2^2$.
\end{proof}

\subsection{Cayley scheme}

Let $R=\{r(X):~X\in\mathcal{S}\}$, where $r(X)=\{(g,xg):~x\in X,~g\in G\}$. From~\cite[Theorem~2.4.16]{CP} it follows that $\mathcal{X}=(G,R)$ is a Cayley scheme over $G$, i.e. an association scheme on $G$ whose automorphism group contains the group of right translations of $G$.

The \emph{thin radical} of an association scheme is the set of all its basis relations of valency~$1$. An association scheme is \emph{regular} if it coincides with its thin radical. The thin radical of any association scheme forms a group with respect to composition (see~\cite[p.30]{CP}). 

\begin{lemm}\label{propertyscheme}
Any association scheme $\mathcal{Y}=(\Omega,S)$ algebraically isomorphic to $\mathcal{X}$ is commutative and satisfies the following properties.
\\
\\
$(B1)$ The thin radical $E$ of $\mathcal{Y}$ is isomorphic to $C_p\times C_p$ and the quotient $\mathcal{Y}_{\Omega/e}$, where $e=\bigcup_{s\in E} s$, is a regular association scheme isomorphic to $C_2\times C_2$.
\\
\\
$(B2)$ For every $i\in I$, there exists a basis relation $r_i\in S$ such that $r_i=r_i^*$, $n_{r_i}=n_{e_i}=p$, where $e_i=\rad(r_i)$, $e_i\subseteq e$, $E_i=\{s\in S:~s\subseteq e_i\}$ is a subgroup of $E$ isomorphic to $C_p$, and $e_i\cap e_j=1_{\Omega}$ whenever $j\neq i$.
\\
\\
$(B3)$ If $t_i=r_iu_i$, $t_j=r_ju_j$ for some $u_i,u_j\in E$, and $t\in S$ then
$$c_{t_it_j}^{t}>0\Leftrightarrow t=r_ku_k~\text{for}~k\in I\setminus \{i,j\}~\text{and some}~u_k\in E.$$
\end{lemm}

\begin{proof}
The commutativity and properties $(B1)-(B3)$ are preserved under algebraic isomorphisms. So it sufficient to check that $\mathcal{X}$ is commutative and satisfies $(B1)-(B3)$. The commutativity of $\mathcal{X}$ follows from the fact that $G$ is abelian. Note that $e=r(P)$ and $\mathcal{X}$ satisfies $(B1)$ because $\mathcal{A}$ satisfies $(A1)$. Due to~$(A2)$, we obtain that $(B2)$ holds for $r_i=r(X_i)$, $i\in I$. Finally, $(B3)$ follows from $(A3)$ and the equalities $e=r(P)$, $r_i=r(X_i)$, $i\in I$.
\end{proof}

Put 
$$R_i=\{r(X):~X\in \mathcal{S}_i\},~i\in I,$$ 
$$R_{ij}=\{r(X):~X\in \mathcal{S}_{ij}\},~(i,j)\in J,$$ 
$$R_0=\{r(X):~X\in \mathcal{S}_0\}.$$ 
Due to~\cite[Theorem~2.4.16]{CP}, we conclude that $\mathcal{X}_i=(G,R_i)$, $i\in I$, $\mathcal{X}_{ij}=(G,R_{ij})$, $(i,j)\in J$, and $\mathcal{X}_0=(G,R_0)$ are Cayley schemes over $G$. Lemma~\ref{fusionring} implies that each of the Cayley schemes $\mathcal{X}_i$, $\mathcal{X}_{ij}$, $\mathcal{X}_0$ is a fusion of $\mathcal{X}$, i.e. every basis relation of each of the above Cayley schemes is a union of some basis relations of $\mathcal{X}$.

Let $\mathcal{Y}=(\Omega,S)$ be an association scheme algebraically isomorphic to $\mathcal{X}$ and $\varphi$ an algebraic isomorphism  from $\mathcal{X}$ to $\mathcal{Y}$. Since each of the Cayley schemes $\mathcal{X}_i$, $\mathcal{X}_{ij}$, $\mathcal{X}_0$ is a fusion of $\mathcal{X}$, one can define the images $\mathcal{Y}_i=\mathcal{X}_i^{\varphi}=(\Omega,S_i)$, $\mathcal{Y}_{ij}=\mathcal{X}_{ij}^{\varphi}=(\Omega,S_{ij})$, $\mathcal{Y}_0=\mathcal{X}_0^{\varphi}=(\Omega,S_0)$ of these Cayley schemes under $\varphi$. In view of~\cite[Corollary~2.3.21]{CP}, each of the above images is an association scheme on $\Omega$. 

\begin{lemm}\label{orderscheme}
In the above notations, the following statements hold.
\\
\\
$(1)$ If $i,j\in I$ such that $i\neq j$ then $S=\{s_i\cap s_j:~s_i\in S_i,~s_j\in S_j\}$.
\\
\\
$(2)$ If $(j,k),(l,m)\in J$ such that $\{j,k\}\cap \{l,m\}=\{i\}$ then $S_i=\{s_{jk}\cap s_{lm}:~s_{jk}\in S_{jk},~s_{lm}\in S_{lm}\}$.
\end{lemm}

\begin{proof}
The statement of the lemma follows from Lemma~\ref{orderring} and~\cite[Proposition~2.3.18~(1)]{CP}. 
\end{proof}

\begin{lemm}\label{fusionscheme}
In the above notations, the following statements hold.
\\
\\
$(1)$ $\mathcal{Y}_i\leq \mathcal{Y}$ and $\mathcal{Y}_i$ is isomorphic to the tensor product of two association schemes each of which is the wreath product of two regular association schemes of degrees~$p$ and~$2$ for every $i\in I$.
\\
\\
$(2)$ $\mathcal{Y}_{ij}\leq \mathcal{Y}_i$ and $\mathcal{Y}_{ij}\leq \mathcal{Y}_i$ for every $(i,j)\in J$.
\\
\\
$(3)$ $\mathcal{Y}_0\leq \mathcal{Y}_{ij}$ for every $(i,j)\in J$ and $\mathcal{Y}_0$ is isomorphic to the wreath product of two regular association schemes of degrees~$p^2$ and~$4$.
\end{lemm}

\begin{proof}
Lemma~\ref{fusionring} implies that Lemma~\ref{fusionscheme} holds if $\varphi$ is trivial, i.e. if $\varphi$ maps every relation to itself (in this case $\mathcal{Y}=\mathcal{X}$). Suppose that $\varphi$ is nontrivial. By the definition of the image of a fusion of the Cayley scheme $\mathcal{X}$ with respect to the algebraic isomorphism $\varphi$, we obtain $\mathcal{Y}_i\leq \mathcal{Y}$ for every $i\in I$, $\mathcal{Y}_{ij}\leq \mathcal{Y}_i$ and $\mathcal{Y}_{ij}\leq \mathcal{Y}_j$ for every $(i,j)\in J$, and $\mathcal{Y}_{0}\leq \mathcal{Y}_{ij}$ for every $(i,j)\in J$. 

By~\cite[Theorem~2.3.33]{CP}, every regular association scheme is separable. The tensor and wreath products of two separable association schemes are separable by~\cite[Corollary~3.2.24]{CP} and~\cite[Theorem~3.4.9]{CP} respectively. So $\mathcal{X}_i$ and $\mathcal{X}_0$ are separable and hence $\mathcal{Y}_i\cong \mathcal{X}_i$ and $\mathcal{Y}_0\cong \mathcal{X}_0$. This proves Statements~1 and~3 of the lemma.
\end{proof}

\section{Proof of Theorem~1}
	
We start with auxiliary lemma.

\begin{lemm}\label{aux}
Let $\mathcal{Y}=(\Omega,S)$ be an association scheme, $e$ a thin parabolic of $\mathcal{Y}$, and $f\in \aut(\mathcal{Y})$. If for every $\Delta\in \Omega/e$ there exists $\delta\in \Delta$ such that $\delta^f=\delta$ then $f=\id_{\Omega}$. 
\end{lemm}

\begin{proof}
Let us prove that $\alpha^f=\alpha$ for every $\alpha\in \Omega$. Suppose that $\Delta$ is the class of $e$, containing $\alpha$. By the condition of the lemma, there exists $\delta\in \Delta$ such that $\delta^f=\delta$. Let $s\in S$ such that $(\delta,\alpha)\in s$. Note that $s^f=s$ because $f\in \aut(\mathcal{Y})$. So $(\delta^f,\alpha^f)=(\delta,\alpha^f)\in s$. Since $\delta$ and $\alpha$ belong to the same class of $e$, the pair $(\delta,\alpha)$ belongs to $e$. Therefore $s\subseteq e$ and hence $n_s=1$.  This implies that $\alpha^f=\alpha$. Thus, $f$ fixes every point from $\Omega$, i.e. $f=\id_{\Omega}$.
\end{proof}

To prove Theorem~1, it is sufficient to prove that the Cayley scheme $\mathcal{X}$ constructed in the previous section is nonschurian and separable. At first, we prove that the automorphism group of an association scheme algebraically isomorphic to $\mathcal{X}$ is regular. We divide the proof of this fact into two lemmas. In the first lemma, we show that the automorphism group of such scheme is semiregular and, in the second lemma, we check that it is transitive and hence regular.

\begin{lemm}\label{proof1}
If $\mathcal{Y}=(\Omega,S)$ is an association scheme algebraically isomorphic to $\mathcal{X}$ then the group $\aut(\mathcal{Y})$ is semiregular. In particular, $|\aut(\mathcal{Y})|\leq 4p^2$.
\end{lemm}

\begin{proof}
Throughout the proof of this lemma, we use the notations from Lemma~\ref{propertyscheme} for $\mathcal{Y}$. Let $\varepsilon\in \Omega$, $f\in \aut(\mathcal{Y})$ such that $\varepsilon^f=\varepsilon$, and $i\in I$. Denote the class of $e_i$ containing $\varepsilon$ by $\Delta_i$. Suppose that $\alpha_i\in \varepsilon r_i$ and $\Lambda_i$ is the class of $e_i$ containing $\alpha_i$. Due to the equality $e_i=\rad(r_i)$, we conclude that $\Delta_i\times \Lambda_i\subseteq r_i$. Together with the equality $n_{r_i}=p$ from $(B2)$, this yields that
$$\varepsilon r_i=\Lambda_i.~\eqno(2)$$ 
Let $\alpha_i^f=\beta_i$. Since $\varepsilon^f=\varepsilon$ and $(r_i)^f=r_i$, Eq.~(2) implies that $\beta_i\in \Lambda_i$ and hence $(\alpha_i,\beta_i)\in s_i$ for some $s_i\in E_i$. The equality $e_i\subseteq e$ from $(B2)$ yields that $n_{s_i}=1$. 

Let $i,j\in I$ such that $i\neq j$. The pair $(\alpha_i,\alpha_j)$ belongs to $r_i^{*}r_j$ because $(\varepsilon,\alpha_i)\in r_i$ and $(\varepsilon,\alpha_j)\in r_j$. Due to $(B2)$, we have $r_i=r_i^*$ and hence $(\alpha_i,\alpha_j)\in r_ir_j$. From $(B3)$ it follows that $(\alpha_i,\alpha_j)\in r_ku\in S$ for $k\in I\setminus \{i,j\}$ and some $u\in E$. Since $f$ is an automorphism of $\mathcal{Y}$, the equality $(r_ku)^f=r_ku$ holds. This implies that
$$(\beta_i,\beta_j)\in r_ku.~\eqno(3)$$
On the other hand, 
$$(\beta_i,\beta_j)\in s_i^{*}r_kus_j~\eqno(4)$$ 
because $(\alpha_i,\beta_i)\in s_i$, $(\alpha_i,\alpha_j)\in r_ku$, and $(\alpha_j,\beta_j)\in s_j$. 
The relation $s_i^{*}r_kus_j$ is basis because each of the relations $s_i^{*}$, $r_ku$, $s_j$ is basis and $n_{s_i^{*}}=n_{s_j}=1$. Therefore from Eqs.~(3) and~(4) it follows that $r_ku=s_i^{*}r_kus_j$. Since $\mathcal{Y}$ is commutative by Lemma~\ref{propertyscheme} and $n_u=1$, this implies that $s_i^{*}s_j\subseteq \rad(r_k)=e_k$ or, equivalently,
$$s_i^{*}s_j\in E_k.~\eqno(5)$$

Assume that $s_i\neq 1_{\Omega}$ for every $i\in I$. Then for every $i\in I$, we have $E_i=\langle s_i \rangle$  because $E_i\cong C_p$ by $(B2)$. So Eq.~(5) yields that 
$$s_1^{*}s_2=s_3^l,~s_2^{*}s_3=s_1^m,~\text{and}~s_3^{*}s_1=s_2^n$$
for some $l,m,n\in \{0,\ldots,p-1\}$. The above equalities imply that $s_1^{m+1}=s_3^{-l+1}\in E_3$ and $s_1^{m-1}=s_2^{-n-1}\in E_2$. Note that $E_1\cap E_2=E_1\cap E_3=\{1_{\Omega}\}$ by the last equality from $(B2)$. So $s_1^{m+1}=s_1^{m-1}=1_{\Omega}$ and hence $s_1^2=1_{\Omega}$. This means that $s_1=1_{\Omega}$ because $E_1\cong C_p$ by $(B2)$ and $p\geq 5$, a contradiction to the assumption. Thus, there exists $i\in I$ such that $s_i=1_{\Omega}$. 

Due to Eq.~(5), we have $s_j\in E_k$ and $s_k\in E_j$, where $\{j,k\}=I\setminus \{i\}$. Together with the last equality from $(B2)$ and the equality $s_i=1_{\Omega}$, this yields that
$$s_1=s_2=s_3=1_{\Omega}.~\eqno(6)$$

In view of Eq.~(6), we obtain $\alpha_i=\beta_i=\alpha_i^f$ for every $i\in I$. Note that $\varepsilon$ and $\alpha_i$ belong to distinct classes of $e$ by the choice of $\alpha_i$ for every $i\in I$. Also $\alpha_1$, $\alpha_2$, and $\alpha_3$ belong to pairwise distinct classes of $e$ because otherwise $c_{r_ir_j}^{t}>0$ for some $i,j\in I$ and $t\in E$ which contradicts to $(B3)$. One can see that $|\Omega/e|=4$ by $(B1)$. So the automorphism $f$ fixes at least one point in every class of $e$. Therefore $f=\id_{\Omega}$ by Lemma~\ref{aux}. Thus, the stabilizer $\aut(\mathcal{Y})_{\varepsilon}$ is trivial for every $\varepsilon\in \Omega$ and hence $\aut(\mathcal{Y})$ is semiregular.
\end{proof}

\begin{lemm}\label{proof2}
In the conditions of Lemma~\ref{proof1}, the group $\aut(\mathcal{Y})$ is regular.
\end{lemm}

\begin{proof}
In view of Lemma~\ref{proof1}, it is sufficient to prove that
$$|\aut(\mathcal{Y})|\geq 4p^2.~\eqno(7)$$
Let $\mathcal{Y}_i$, $\mathcal{Y}_{ij}$, and $\mathcal{Y}_0$ be the association schemes defined in Section~2.2. Statement~1 of Lemma~\ref{fusionscheme} and~\cite[Theorem~3.2.21, Theorem~3.4.6]{CP} imply that the group $\aut(\mathcal{Y}_i)$ is the direct product of two permutation groups each of which is the wreath product of two regular groups of degrees~$p$ and~$2$ for every $i\in I$. So
$$|\aut(\mathcal{Y}_i)|=4p^4.~\eqno(8)$$
From Statement~3 of Lemma~\ref{fusionscheme} and~\cite[Theorem~3.4.6]{CP} it follows that 
$$\aut(\mathcal{Y}_0)\geq \aut(\mathcal{Y}_{ij})~\eqno(9)$$
for every $(i,j)\in J$ and the group $\aut(\mathcal{Y}_0)$ is the wreath product of two regular groups of degrees~$p^2$ and~$4$. So
$$|\aut(\mathcal{Y}_0)|=4p^8.~\eqno(10)$$
Statement~1 and~2 of Lemma~\ref{orderscheme} imply that
$$\aut(\mathcal{Y}_i)\cap \aut(\mathcal{Y}_j)=\aut(\mathcal{Y})~\eqno(11)$$
whenever $i\neq j$ and
$$\aut(\mathcal{Y}_{jk})\cap \aut(\mathcal{Y}_{lm})=\aut(\mathcal{Y}_i)~\eqno(12)$$
whenever $\{j,k\}\cap\{l,m\}=\{i\}$ respectively. 

From Eqs.~(8), (9), (10), and~(12) it follows that 
$$|\aut(\mathcal{Y}_{12})||\aut(\mathcal{Y}_{13})|\leq |\aut(\mathcal{Y}_1)||\aut(\mathcal{Y}_0)|=16p^{12}.$$
This implies that at least one of the numbers $|\aut(\mathcal{Y}_{12})|$, $|\aut(\mathcal{Y}_{13})|$ does not exceed~$4p^6$. Without loss of generality, let
$$|\aut(\mathcal{Y}_{12})|\leq 4p^6.~\eqno(13)$$
Now in view of Eqs.~(8), (11), (12), and~(13), we conclude that
$$|\aut(\mathcal{Y})|\geq |\aut(\mathcal{Y}_1)| |\aut(\mathcal{Y}_2)|/|\aut(\mathcal{Y}_{12})|\geq 4p^2.$$
Thus, Eq~(7) holds.
\end{proof}

From Lemma~\ref{proof2} it follows that $\aut(\mathcal{X})$ is regular. In particular, $|\aut(\mathcal{X})|=4p^2$. On the other hand, $|r_1|=|r_2|=|r_3|=|G||n_{r_1}|=4p^3$. So $r_1$, $r_2$, $r_3$ are not $2$-orbits of the group $\aut(\mathcal{X})$ and hence $\mathcal{X}$ is nonschurian.

Let $\mathcal{Y}=(\Omega,S)$ be an association scheme and $\varphi$ an algebraic isomorphism from $\mathcal{X}$ to $\mathcal{Y}$. Lemma~\ref{proof2} implies that $\aut(\mathcal{Y})$ is regular. So we may assume that $\mathcal{Y}$ is a Cayley scheme over a group $H$ of order~$4p^2$. 

Let $\mathcal{T}=\{X(s):~s\in S\}$, where $X(s)=\{x\in H:~(1_{H},x)\in s\}$ and $1_H$ is the identity of $H$. From~\cite[Theorem~2.4.16]{CP} it follows that the partition $\mathcal{T}$ defines the $S$-ring $\mathcal{B}$ over $H$. The bijecton $\psi$ from $\mathcal{S}$ to $\mathcal{T}$ defined by the equality 
$$X^{\psi}=X(r^{\varphi}(X))$$
is an algebraic isomorphism from $\mathcal{A}$ to $\mathcal{B}$ in the sense of~\cite{Ry}. To prove that $\varphi$ is induced by a combinatorial isomorphism from $\mathcal{X}$ to $\mathcal{Y}$, it is sufficient to prove that $\psi$ is induced by a Cayley isomorphism from $\mathcal{A}$ to $\mathcal{B}$, i.e. by a group isomorphism $f$ from $G$ to $H$ for which $\mathcal{S}^f=\mathcal{T}$, where $\mathcal{S}^f=\{X^f:~X\in \mathcal{S}\}$. This follows from the next two lemmas.

\begin{lemm}\label{proof3}
In the above notations, $H\cong C_{2p}\times C_{2p}\cong G$.
\end{lemm}

\begin{proof}
The algebraic isomorphism $\psi$ induces the bijection between $\mathcal{A}$-sections of $G$ and $\mathcal{B}$-sections of $H$. Moreover, $\psi$ maps an $\mathcal{A}$-section to a $\mathcal{B}$-section of the same order (see~\cite[p.~5]{Ry}). So $P^{\psi}$ is a $\mathcal{B}$-subgroup of order $p^2$. The number $m$ of Sylow $p$-subgroups of $H$ divides $|H|=4p^2$ and $m\equiv 1\mod~p$ by the third Sylow theorem. Since $p\geq 5$, we conclude that $m=1$ and hence $P^{\psi}$ is a unique Sylow $p$-subgroup of $H$. So $P^{\psi}$ is normal in $H$ by the second Sylow theorem. 

The algebraic isomorphism $\psi$ induces the algebraic isomorphisms from $\mathcal{A}_P$ to $\mathcal{B}_{P^{\psi}}$ and from $\mathcal{A}_{G/P}$ to $\mathcal{B}_{H/P^{\psi}}$. Since $\mathcal{A}_P=\mathbb{Z}P$, $\mathcal{A}_{G/P}=\mathbb{Z}(G/P)$, and every algebraic isomorphism preserves a rank of an $S$-ring, $\mathcal{B}_{P^{\psi}}=\mathbb{Z}P^{\psi}$ and $\mathcal{B}_{H/P^{\psi}}=\mathbb{Z}(H/P^{\psi})$. Now from~\cite[Corollary~2.3.34]{CP} it follows that
$$P^{\psi}\cong P\cong C_p\times C_p,~H/P^{\psi}\cong G/P\cong C_2\times C_2.~\eqno(14)$$

All groups of order $4p^2$, where $p\geq 5$, are listed in~\cite[Section~IV]{Il}. In accordance with this list and Eq.~(14), $H$ is isomorphic to one of the following groups:
$$C_{2p}\times C_{2p},~C_{2p}\times D_{2p},~D_{2p}\times D_{2p},~(C_p^2\rtimes C_2)\times C_2,$$ where $D_{2p}$ is the dihedral group of order~$2p$ and a nontrivial element of $C_2$ inverts every element of $C_p^2$ in the latter group.

For a given subset $Z$ of a group and an integer $m$, denote by $k_m(Z)$ the number of elements of order~$m$ in $Z$. The sets $X_1$, $X_2$, $X_3$ are the only nontrivial inverse-closed basic sets of $\mathcal{A}$. Therefore $X_1^{\psi}$, $X_2^{\psi}$, $X_3^{\psi}$ are the only nontrivial inverse-closed basic sets of $\mathcal{B}$. This implies that every element of $H$ of order~$2$ lies in $X_1^{\psi}\cup X_2^{\psi} \cup X_3^{\psi}$. Since $|X_i^{\psi}|=|X_i|=p$ for every $i\in I$, we conclude that 
$$k_2(H)=k_2(X_1^{\psi})+k_2(X_2^{\psi})+k_2(X_3^{\psi})\leq 3p.~\eqno(15)$$
If $H\cong D_{2p}\times D_{2p}$ or $H\cong (C_p^2\rtimes C_2)\times C_2$ then $k_2(H)\geq p^2$. In view of $p\geq 5$, we obtain a contradiction to Eq.~(15). Thus, $H\cong C_{2p}\times C_{2p}$ or $H\cong C_{2p}\times D_{2p}$.

We are done if $H\cong C_{2p}\times C_{2p}$. Assume that $H=H_1\times H_2$ where $H_1\cong C_{2p}$ and $H_2\cong D_{2p}$. In this case
$$k_2(H)=2p+1.~\eqno(16)$$
For every $i\in I$, put $U_i=\langle X_i\rangle$. Note that $U_i\cong C_{2p}$ for every $i\in I$ and $U_i\cap U_j=\{1_G\}$ whenever $i\neq j$. So the group $U_i^{\psi}=\langle X_i^{\psi}\rangle$ is of order~$2p$ and hence it is isomorphic to one of the groups $C_{2p}$, $D_{2p}$. Also $U_i^{\psi}\cap U_j^{\psi}=\{1_H\}$ whenever $i\neq j$. If $U_i^{\psi}\cong U_j^{\psi}\cong C_{2p}$ for distinct $i,j\in I$ then 
$$k_2(H)=k_2(U_1^{\psi})+k_2(U_2^{\psi})+k_2(U_3^{\psi})\leq p+2,$$
a contradiction to Eq.~(16). Therefore there are $i,j\in I$ such that
$$i\neq j,~U_i^{\psi}\cong U_j^{\psi}\cong D_{2p},~\text{and}~U_i^{\psi}\cap U_j^{\psi}=\{1_H\}.~\eqno(17)$$

Clearly, $H_1\leq Z(H)$. So $H_1\cap U_i^{\psi}=H_1\cap U_j^{\psi}=\{1_H\}$ because the center of $D_{2p}$ is trivial. This implies that $|N_H(U_i^{\psi})|\geq |H_1U_i^{\psi}|=4p^2$ and $|N_H(U_j^{\psi})|\geq |H_1U_j^{\psi}|=4p^2$. Therefore $U_i^{\psi}$ and $U_j^{\psi}$ are normal in $H$. Together with Eq.~(17), this yields that $H=U_i^{\psi}\times U_j^{\psi}\cong D_{2p}\times D_{2p}$, which contradicts the assumption.
\end{proof}

\begin{lemm}\label{proof4}
In the above notations, $\psi$ is induced by a Cayley isomorphism.

\end{lemm}

\begin{proof}
Let $i\in I$. The set $X_i$ is a coset by an $\mathcal{A}$-subgroup $P_i$ of order~$p$. So $X_i^{\psi}$ is a coset by a $\mathcal{B}$-subgroup $P_i^{\psi}$ of order~$p$ by~\cite[Eq.~(3)]{Ry}. Since $U_i^{\psi}\cong C_{2p}$, the set $X_i^{\psi}$ contains a unique element of order~$2$. Denote this element by $b_i$. Clearly, $X_i^{\psi}=P_i^{\psi}b_i$ and $U_i^{\psi}=P_i^{\psi}\times \langle b_i^{\psi}\rangle$. 

Let $f_0$ be a bijection from $P$ to $P^{\psi}$ such that $\{g\}^{\psi}=\{g^{f_0}\}$ for every $g\in P$. Let us show that $f_0$ is an isomorphism from $P$ to $P^{\psi}$. Let $x,y\in P$. Clearly, $c_{\{x\}\{y\}}^{\{xy\}}=1$. Since $\psi$ is an algebraic isomorphism, we have $c_{\{x^{f_0}\}\{y^{f_0}\}}^{\{(xy)^{f_0}\}}=c_{\{x\}^{\psi}\{y\}^{\psi}}^{\{xy\}^{\psi}}=c_{\{x\}\{y\}}^{\{xy\}}=1$. This implies that $(xy)^{f_0}=x^{f_0}y^{f_0}$ and we are done.

Let $f$ be an isomorphism from $G$ to $H$ defined as follows:
$$f^P=f_0,~a_i^f=b_i~\text{for every}~i\in I.$$
Let us prove that $X^f=X^{\psi}$ for every $X\in \mathcal{S}(\mathcal{A})$. If $X\subseteq P$ then $X=\{g\}$ for some $g\in P$ and $X^f=X^{f_0}=\{g^{f_0}\}=\{g\}^{\psi}=X^{\psi}$. If $X=X_i$ for some $i\in I$ then
$$X^f=(P_ia_i)^f=P_i^{f_0}b_i=P_i^{\psi}b_i=X^{\psi}.$$
It remains to consider the case when $X=X_ig$ for some $i\in I$ and $g\in P$. In this case $X$ is a unique basic set of $\mathcal{A}$ such that $c_{X_i\{g\}}^X>0$. So $X^{\psi}$ is the unique basic set of $\mathcal{B}$ such that $c_{X_i^{\psi}\{g\}^{\psi}}^{X^{\psi}}>0$. This implies that $X^{\psi}=X_i^{\psi}\{g\}^{\psi}=X_i^fg^f=X^f$ as required. Thus, $f$ induces~$\psi$.
\end{proof}

We have proved that $\psi$ is induced by a Cayley isomorphism. So $\varphi$ is induced by a combinatorial isomorphism. Thus, every algebraic isomorphism from $\mathcal{X}$ is induced by a combinatorial isomorphism, i.e. $\mathcal{X}$ is separable.

\end{document}